\documentclass{amsart}

\usepackage{latexsym}
\usepackage{amsfonts,amssymb} 
\usepackage{graphicx}

\def\cal{\mathcal}

\newcommand{\field}[1]{\mathbb{#1}}

\newcommand{\C}{\field{C}}

\newcommand{\T}{\field{T}}
\newcommand{\Z}{\field{Z}}

\newcommand{\kk}{{\bf k}}

\newtheorem{theorem}{Theorem}[section]
\newtheorem{proposition}{Proposition}[section]
\newtheorem{lemma}{Lemma}[section]
\newtheorem{corollary}{Corollary}[section]
\newtheorem{definition}{Definition}[section]
\newtheorem{remark}{Remark}[section]

\newtheorem{example}{Example}[section]

\begin{document}

\makeatletter	   
\makeatother     

\title{Laurent cancellation for rings of transcendence degree one over a field}
\author{Gene Freudenburg}
\date{\today} 
\maketitle

\pagestyle{plain}

\begin{abstract} If $R$ is an integral domain and $A$ is an $R$-algebra, then $A$ has the {\it Laurent cancellation property over $R$} if 
$A^{[\pm n]}\cong_RB^{[\pm n]}$ implies $A\cong_RB$ ($n\ge 0$ and $B$ an $R$-algebra). Here, $A^{[\pm n]}$ denotes the ring of Laurent polynomials in $n$ variables over $A$. 
Our main result  ({\it Thm.~\ref{thm-2}}) is that, if $R$ is a field and the transcendence degree of $A$ over $R$ is one, then $A$ has the Laurent cancellation property over $R$. 
The proof uses the characterization of Laurent polynomial rings over a field given in {\it Thm.~\ref{character}}.
 \end{abstract}
 
\section{Introduction}

If $R$ is an integral domain, then an {\it $R$-algebra} will mean an integral domain containing $R$ as a subring. If $A$ is an $R$-algebra and $n\in\Z$, $n\ge 0$, then $A^{[n]}$ is the polynomial ring in $n$ variables over $A$, and $A^{[\pm n]}$ is the ring of Laurent polynomials in $n$ variables over $A$.

In this paper, we consider the following question. 
\medskip
\begin{quote} {\it Let $R$ be an integral domain, let $A$ and $B$ be $R$-algebras, and let $n$ be a non-negative integer. Does $A^{[\pm n]}\cong_RB^{[\pm n]}$ imply $A\cong_RB$?}
\end{quote}
\medskip
We say that $A$ has the {\it Laurent cancellation property over $R$} if this question has a positive answer for all pairs $(B,n)$. 

Our main result ({\it Thm.~\ref{thm-2}}) is that, if $\kk$ is a field, then any $\kk$-algebra of transcendence degree one over $\kk$ has the Laurent cancellation property over $\kk$. This result parallels the well-known theorem of Abhyankar, Eakin, and Heinzer which asserts that, if $A$ and $B$ are $\kk$-algebras of transcendence one over $\kk$, then the condition $A^{[n]}\cong_{\kk}B^{[n]}$ for some $n\ge 0$ implies $A\cong_{\kk}B$; see \cite{Abhyankar.Eakin.Heinzer.72}. Note that our main result implies that if $X$ and $Y$ are algebraic curves over $\kk$, 
and if $\T^n$ is the torus of dimension $n$ over $\kk$ ($n\ge 0$), then the condition $X\times\T^n\cong Y\times\T^n$ implies $X\cong Y$. 

In \cite{Makar-Limanov.98}, Makar-Limanov gives a proof of the Abhyankar-Eakin-Heinzer theorem for the field $\kk =\C$ using the theory of locally nilpotent derivations (LNDs); see also \cite{Crachiola.Makar-Limanov.05}, Cor. 3.2. The proof of our main result uses $\Z$-gradings in a similar way. Where Makar-Limanov uses the subring of elements of degree zero for all LNDs, we use the subring of elements of degree zero for all $\Z$-gradings over $R$, denoted ${\cal N}_R(A)$. The other key ingredient in the proof is {\it Thm.~\ref{character}}, which is the following characterization of Laurent polynomial rings over a field. 
\begin{quote} {\it 
Let $\kk$ be a field, and let $A$ be a $\kk$-algebra. The following are equivalent.
\smallskip
\begin{itemize}
\item [1.] $A=\kk^{[\pm 1]}$.
\smallskip
\item [2.] The following three conditions hold. 
\subitem {\rm (a)} $\kk$ is algebraically closed in $A$
\subitem {\rm (b)} ${\rm tr.deg}_{\kk}A=1$
\subitem {\rm (c)} $A^*\not\subset {\cal N}_{\kk}(A)$
\end{itemize}
}
\end{quote}

Certain cases of Laurent cancellation were investigated in \cite{Freudenburg.ppta}; see {\it Remark~\ref{summary}}. In addition, Bhatwadekar and Gupta \cite{Bhatwadekar.Gupta.12} have shown that the Laurent polynomial ring $R^{[\pm n]}$ has the Laurent cancellation property over $R$; see {\it Thm.~\ref{thm-4}} below. We are not aware of an example of an integral domain $R$ and an $R$-algebra $A$ such that $A$ fails to have the Laurent cancellation property over $R$.

This paper was motivated by a question of David Speyer, who posed the question of Laurent cancellation above in the case $A$ and $B$ are cluster algebras. 

\subsection{Terminology and Notation}

The group of units of the integral domain $A$ is denoted $A^*$, and the field of fractions of $A$ is ${\rm frac}(A)$. Given $f\in A$, $A_f$ denotes the localization of $A$ at $f$. Given $z\in A^*$, the notation $z^{\pm 1}$ is used for the set $\{ z,z^{-1}\}$. 

For $n\ge 0$, the polynomial ring in $n$ variables over $A$ is denoted by $A^{[n]}$. If $A[x_1,...,x_n]=A^{[n]}$, the ring of Laurent polynomials over $R$ is the subring of ${\rm frac}(A^{[n]})$ defined and denoted by:
\[
 A^{[\pm n]}=A[x_1,x_1^{-1},...,x_n,x_n^{-1}]
\]

For any subring $S\subset A$, the transcendence degree of $A$ over $S$ is equal to the transcendence degree of ${\rm frac}(A)$ over ${\rm frac}(S)$, denoted 
${\rm tr.deg}_SA$. The set of elements in $A$ algebraic over $S$ is denoted by ${\rm Alg}_SA$; we also say that ${\rm Alg}_SA$ is the {\it algebraic closure of $S$ in $A$}. If $S={\rm Alg}_SA$, then $S$ is {\it algebraically closed in $A$}. 
Any $\Z$-grading of $A$ such that $S\subset A_0$ is a {\it $\Z$-grading over $S$}, where $A_0$ denotes the subring of elements of degree zero. 

If $R$ is an integral domain, $A$ is an $R$-algebra, and $W$ is a subset of $A$, then $R[W]$ is the $R$-subalgebra of $A$ generated by $W$. 

\section{$\Z$-Gradings and the Neutral Subalgebra}

\subsection{$\Z$-Gradings} Assume that $R$ is an integral domain, and $A$ is an $R$-algebra.
The set of $\Z$-gradings of $A$ is denoted $A(\Z )$, and the subset of $\Z$-gradings of $A$ over $R$ is denoted $A(\Z ,R)$. Given $\frak{g}\in A(\Z )$, let $\deg_{\frak{g}}$ denote the induced degree function on $A$, and let $A=\oplus_{i\in\Z}A_i$ be the decomposition of $A$ into $\mathfrak{g}$-homogeneous summands, where $A_i$ consists of $\mathfrak{g}$-homogeneous elements of degree $i$. Define $A^{\frak{g}}=A_0$, which is a subalgebra of $A$. 
The subalgebra $S\subset A$ is {\it $\mathfrak{g}$-homogeneous} if $S$ is generated by $\mathfrak{g}$-homogeneous elements.

Given $a\in A$, write $a=\sum_{i\in\Z}a_i$, where $a_i\in A_i$ for each $i$. The {\it support} of $a$ relative to $\mathfrak{g}$ is defined by:
 \[
{\rm Supp}_{\mathfrak{g}}(a)=\{ i\in\Z\,| \, a_i\ne 0\}
\] 
Note that (i) ${\rm Supp}_{\mathfrak{g}}(a)=\emptyset$ if and only if $a=0$, and (ii) $\#{\rm Supp}_{\mathfrak{g}}(a)=1$ if and only if $a$ is non-zero and homogeneous. 

\begin{lemma}\label{homogeneous} Let $A$ be an integral domain, and let $\mathfrak{g}\in A(\Z )$ be given.
\begin{itemize}
\item [{\bf (a)}] $A^{\mathfrak{g}}$ is algebraically closed in $A$. 
\medskip
\item [{\bf (b)}]  If $H\subset A$ is a $\mathfrak{g}$-homogeneous subalgebra, then ${\rm Alg}_HA$ is also a $\mathfrak{g}$-homogeneous subalgebra.
\end{itemize}
\end{lemma}

\begin{proof} Let $A=\oplus_{i\in\Z}A_i$ be the decomposition of $A$ into $\mathfrak{g}$-homogeneous summands. Given $a\in A$, write $a=\sum_{i\in\Z}a_i$, where $a_i\in A_i$, and let $\bar{a}$ denote the highest-degree (non-zero) homogeneous summand of $a$. 

In order to prove part (a), let $v\in A$ be algebraic over $A^{\mathfrak{g}}$. If $v\not\in A^{\mathfrak{g}}$, then we may assume that $\deg_{\mathfrak{g}}v >0$. Suppose that $\sum_{0\le i\le n}c_iv^i=0$ is a non-trivial dependence relation for $v$ over $A^{\mathfrak{g}}$, where $c_i\in A^{\mathfrak{g}}$ for each $i$, and $n\ge 1$.
Since $\deg_{\mathfrak{g}}\bar{v} >0$ and $\deg_{\mathfrak{g}}c_i=0$ for each $i$, we see that $c_n\bar{v}^n=0$, a contradiction. Therefore, $v\in A^{\mathfrak{g}}$, and $A^{\mathfrak{g}}$ is algebraically closed in $A$. 

For part (b), given an integer $n\ge 0$, let $H(n)$ denote the ring obtained by adjoining to $H$ all elements $a\in {\rm Alg}_HA$ such that $\#{\rm Supp}_{\mathfrak{g}}(a)\le n$. In particular, $H(0)=H$. We show by induction on $n$ that, for each $n\ge 1$:
\begin{equation}\label{decomp}
H(n) \subset H(1)
\end{equation}
This property implies ${\rm Alg}_HA=H(1)$, which is a $\mathfrak{g}$-homogeneous subring of $A$.

Assume that, for some $n\ge 2$, $H(n-1)\subset H(1)$. Let $a\in {\rm Alg}_HA$ be given such that $\#{\rm Supp}_{\mathfrak{g}}(a)=n$, and let
$\sum_{i\ge 0}h_ia^i=0$ be a non-trivial dependence relation for $a$ over $H$, where $h_i\in H$ for each $i$. Define:
\[
d=\max_{i\ge 0}\{ \deg_{\mathfrak{g}} h_ia^i\} \quad {\rm and}\quad I=\{ i\in\Z\, |\, i\ge 0\,\, ,\,\,  \deg_{\mathfrak{g}} h_ia^i=d\}
\]
Then $I$ is non-empty, and $\sum_{i\in I}\bar{h_i}\bar{a}^i=0$. Since $H$ is homogeneous, $\bar{h_i}\in H$ for each $i$. Therefore, $\bar{a}$ is algebraic over $H$. Since $a=(a-\bar{a})+\bar{a}$, it follows that $a\in H(n-1)+H(1)$. Since $H(n-1)\subset H(1)$ by the induction hypothesis, we see that $a\in H(1)$, thus proving by induction the equality claimed in (\ref{decomp}). 
\end{proof}

\subsection{The Neutral Subalgebra}

Assume that $R$ is an integral domain, and $A$ is an $R$-algebra.

\begin{definition} {\rm 
The {\it neutral $R$-subalgebra} of $A$ is:
\[
{\cal N}_R(A) = \cap_{\frak{g}\in A(\Z ,R)}A^{\frak{g}}
\]
The  {\it neutral subring} of $A$ is:
\[
{\cal N}(A) = \cap_{\frak{g}\in A(\Z )}A^{\frak{g}}
\]
$A$ is a {\it neutral $R$-algebra} if ${\cal N}_R(A)=A$. $A$ is a {\it neutral ring} if ${\cal N}(A)=A$. }
\end{definition}

\begin{lemma}\label{neutral} Let $R$ be an integral domain, and let $A$ be an $R$-algebra.
\begin{itemize}
\item [{\bf (a)}] ${\cal N}_R(A)$ is algebraically closed in $A$
\medskip
\item [{\bf (b)}] ${\cal N}_R(A^{[n]})\subset {\cal N}_R(A)$ and ${\cal N}_R(A^{[\pm n]})\subset {\cal N}_R(A)$ for each $n\ge 0$
\medskip
\item [{\bf (c)}] If $A$ is algebraic over $R[A^*]$, then ${\cal N}_R(A^{[n]}) = {\cal N}_R(A)$ and ${\cal N}_R(A^{[\pm n]})={\cal N}_R(A)$ for each $n\ge 0$
\end{itemize}
\end{lemma}

\begin{proof} Part (a) is implied by {\it Lemma~\ref{homogeneous}(a)}.

For part (b), let $C=A[y_1^{\pm 1},...,y_n^{\pm 1}]=A^{[\pm n]}$, and let $f\in {\cal N}_R(C)$ be given. Define $\mathfrak{g}\in C(\Z ,A)$ by setting $\deg_{\mathfrak{g}}y_i=1$ for each $i$. If $f\not\in A$, then $\deg_{\mathfrak{g}}f\ne 0$, a contradiction. Therefore, $f\in A$. Suppose that there exists $\mathfrak{h}\in A(\Z ,R)$ such that $\deg_{\mathfrak{h}}f\ne 0$. Then $\mathfrak{h}$ extends to $C$ by setting $\deg_{\mathfrak{h}}y_i=0$ for each $i$, meaning $f\not\in {\cal N}_R(C)$, again a contradiction. Therefore, $f\in {\cal N}_R(A)$. The argument is the same if $C=A^{[n]}$. 

For part (c), let $C=A[y_1^{\pm 1},...,y_n^{\pm 1}]=A^{[\pm n]}$, and let $f\in A$ be given. Suppose that $\mathfrak{g}\in C(\Z ,R)$ has $\deg_{\mathfrak{g}}f\ne 0$. Since every element of $A^*$ is $\mathfrak{g}$-homogeneous, it follows from {\it Lemma~\ref{homogeneous}(b)} that $A$ is a $\mathfrak{g}$-homogeneous subring of $C$. Therefore, $\mathfrak{g}$ restricts to an element of $A(\Z ,R)$ for which the degree of $f$ is non-zero, meaning that $f$ is not in ${\cal N}_R(A)$. 
The argument is the same if $C=A^{[n]}$. 
\end{proof}

\begin{example} {\rm 
Let $R$ be an integral domain, and define $A=R[x,y]/(x^2-y^3-1)$. Then ${\cal N}_R(A)=A$. To see this, let $\mathfrak{g}\in A(\Z ,R)$ be given. Set $K={\rm frac}(R)$ and define $A_K=K\otimes_RA$. 
Then $\mathfrak{g}$ extends to $A_K$, which is the coordinate ring of the plane curve $C:x^2-y^3=1$ over $K$. This $\Z$-grading induces an action of the torus 
$\T ={\rm Spec}(K[t,t^{-1}])$ on $A_K$, namely, if $a\in A_K$ is homogeneous of degree $d$, then $t\cdot a=t^da$. If this were a non-trivial action, then $C$ would contain $\T$ as a dense open orbit, implying that $C$ is $K$-rational, which it is not. Therefore, $\mathfrak{g}$ must be the trivial $\Z$-grading.}
\end{example}


\section{Laurent Polynomial Rings}

\subsection{Units and Automorphisms}

\begin{lemma}\label{units} Let $R$ be an integral domain, and let $A=R[y_1^{\pm 1},...,y_n^{\pm 1}]=R^{[\pm n]}$.
\medskip
\begin{itemize}
\item [{\bf (a)}] $A^*=R^*\cdot\{ y_1^{d_1}\cdots y_n^{d_n}\,|\, d_i\in\Z, i=1,...,n\}=R^*\cdot\Z^n$
\medskip
\item [{\bf (b)}]  Given $E\in SL_n(\Z)$, the $R$-morphism $\phi_E: A\to A$ given by
\[
\phi_E(y_i)=\prod_{1\le j\le n}y_j^{e_{ij}}
\]
where $E=(e_{ij})$, defines an action of $SL_n(\Z )$ on $A$ by $R$-automorphisms.
\medskip
\item [{\bf (c)}] Given $a=(a_1,...,a_n)\in (R^*)^n$, the $R$-morphism $\psi_a:A\to A$ given by
\[
\psi_a(y_i)=a_iy_i
\]
defines an action of $(R^*)^n$ on $A$ by $R$-automorphisms.
\end{itemize}
\end{lemma}

\begin{proof} (a) By induction, it suffices to prove part (a) for the case $n=1$. 

Suppose $A=R[y,y^{-1}]=R^{[\pm 1]}$, and let $u\in A^*$ be given. Write $u=p(y)/y^k$ and $u^{-1}=q(y)/y^l$, where $p,q\in R[y]=R^{[1]}$; $p(0)\ne 0$ and $q(0)\ne 0$;  and $k,l\ge 0$. We thus have $p(y)q(y)=y^{k+l}$. If $k+l>0$, then $p(0)q(0)=0$, contradicting the fact that $A$ is an integral domain. Therefore, $k+l=0$, meaning that $p(y)q(y)=1$ in $R[y]$. 
Since $R[y]^*=R^*$, we see that $p(y)\in R^*$. This proves part (a)

(b) It is easy to check that, for $E,F\in SL_n(\Z)$, $\phi_{EF}=\phi_E\phi_F$.

(c) It is easy to check that, for $a,b\in (R^*)^n$, $\psi_{ab}=\psi_a\psi_b$. 
\end{proof}

\subsection{A Criterion for Cancellation}

\begin{proposition}\label{key1} Let $R$ be an integral domain, and let $A$ and $B$ be $R$-algebras. 
Assume that, for some $n\ge 0$, there exists an $R$-isomorphism
\[
F: A^{[\pm n]}\to B^{[\pm n]}
\]
such that $F(A^*)\subset B$. Then $A\cong_RB$. 
\end{proposition}

\begin{proof}  Let 
\[
C=A[y_1^{\pm 1},...,y_n^{\pm 1}]=A^{[\pm n]} \quad {\rm and}\quad D=B[z_1^{\pm 1},...,z_n^{\pm 1}]=B^{[\pm n]}
\]
By the preceding lemma, we have:
\[
C^*= A^*\cdot\{ y_1^{d_1}\cdots y_n^{d_n}\,|\, d_i\in\Z, i=1,...,n\}
 \quad {\rm and}\quad 
D^* = B^*\cdot\{ z_1^{e_1}\cdots z_n^{e_n}\,|\, e_i\in\Z, i=1,...,n\}
\]
Thus, given $i$ with $1\le i\le n$, there exist $b_i\in B^*$ and $e_{ik}\in\Z$ such that:
\[
F(y_i)=b_i\prod_{1\le k\le n}z_k^{e_{ik}}
\]
Likewise, there exist $a_i\in A^*$ and $d_{ij}\in\Z$ such that:
\[
F^{-1} (z_i)=a_i\prod_{1\le j\le n}y_j^{d_{ij}}
\]
Therefore, given $i$, we have:
\begin{eqnarray*}
z_i &=& FF^{-1}(z_i)\\
&=& F \left( a_i\prod_ky_k^{d_{ik}}\right) \\
&=& F(a_i)\prod_kF(y_k)^{d_{ik}} \\
&=& F(a_i)\prod_k\left( b_k\prod_jz_j^{e_{kj}}\right)^{d_{ik}} \\
&=& \left( F(a_i)\prod_kb_k^{d_{ik}}\right)\prod_k\prod_jz_j^{d_{ik}e_{kj}} \\
&=& \left( F(a_i)\prod_kb_k^{d_{ik}}\right)\prod_jz_j^{(\sum_kd_{ik}e_{kj})}
\end{eqnarray*}

Since $F(a_i)\prod_kb_k^{d_{ik}}\in B$ for each $i$, we conclude that, for each $i,j$ ($1\le i,j\le n$):
\[
F(a_i)\prod_kb_k^{d_{ik}}=1 \quad {\rm and}\quad \sum_kd_{ik}e_{kj}=\delta_{ij}
\]
It follows that, if $E$ is the $n\times n$ matrix $E=(e_{ij})$, then $E\in SL_n(\Z)$.  

Define $b=(b_1,...,b_n)\in (B^*)^n$. Then for each $i=1,...,n$ we see that:
\[
Z_i:=F(y_i)=\phi_E\psi_b(z_i)
\]
On one hand:
\[
D=\phi_E\psi_b(D)=B[Z_1^{\pm 1},...,Z_n^{\pm 1}]
\]
On the other hand:
\[
D=F(C)=F(A)[F(y_1)^{\pm 1},...,F(y_n)^{\pm 1}] = F(A) [ Z_1^{\pm 1},...,Z_n^{\pm 1}]
\]
Therefore, if $I\subset D$ is the ideal $I=(Z_1-1,...,Z_n-1)$, then:
\[
A\cong_RF(A)\cong_RD/I\cong_RB
\]
\end{proof}

Note that, if $A^*=R^*$, then  {\it Thm.~\ref{key1}} implies that $A$ has the Laurent cancellation property over $R$. In particular, every polynomial ring $A=R^{[n]}$ has the Laurent cancellation property over $R$. 

\subsection{A Characterization of Laurent Polynomial Rings over a Field} 

\begin{theorem}\label{dim-one} Let $R$ be an integral domain, and let $A$ be an $R$-algebra such that $R$ is algebraically closed in $A$, ${\rm tr.deg}_RA=1$, and $A^*\not\subset {\cal N}_R(A)$. 
\begin{itemize}
\item  [{\bf (a)}] There exists $u\in A^*$ such that $R[A^*]=R[u,u^{-1}]=R^{[\pm 1]}$.  
\medskip
\item [{\bf (b)}] There exist $r\in R$ and $w\in A_r^*$ such that 
$A_r=R_r[w,w^{-1}]=R_r^{[\pm 1]}$. 
\end{itemize}
\end{theorem}

\begin{proof}
Let $u\in A^*$ be given such that $u\not\in {\cal N}_R(A)$. By {\it Lemma~\ref{neutral}(a)}, we see that $R = {\cal N}_R(A)$ and $R[u]=R^{[1]}$.

Let $\mathfrak{g}\in A(\Z ,R)$ be such that $u\not\in A^{\mathfrak{g}}$, and let $A=\oplus_{i\in\Z}A_i$ be the decomposition of $A$ into $\mathfrak{g}$-homogeneous summands. 
Since $R$ and $A^{\mathfrak{g}}$ are algebraically closed in $A$, it must be that either $A^{\mathfrak{g}}=R$  or $A^{\mathfrak{g}}=A$, and therefore $A^{\mathfrak{g}}=R$. 

If $R[A^*]=R[u,u^{-1}]$, there is nothing further to prove. So assume that $R[A^*]$ is strictly larger than $R[u,u^{-1}]$.
Let $v\in A^*$ be such that $v\not\in R[u,u^{-1}]$. Then $v$ is $\mathfrak{g}$-homogeneous, and $v\not\in A^{\mathfrak{g}}=R$. Since $v$ is algebraic over $R[u,u^{-1}]$, there exists
a non-trivial dependence relation $P(u,v)=0$, where $P\in R[x,y]=R^{[2]}$. 

Set $d=\gcd (\deg_{\mathfrak{g}}u,\deg_{\mathfrak{g}}v)$, and let $a,b\in\Z$ be such that:
\[
\deg_{\mathfrak{g}}u=ad\quad {\rm and}\quad \deg_{\mathfrak{g}}v=bd
\]
If $a<0$, replace $u$ by $u^{-1}$, and if $b<0$, replace $v$ by $v^{-1}$. In this way, we may assume that $a,b>0$. Define a $\Z$-grading $\mathfrak{h}$ of $R[x,y]$ over $R$ by setting $\deg_{\mathfrak{h}}x=a$ and $\deg_{\mathfrak{h}}y=b$. Then it suffices to assume that $P(x,y)$ is $\mathfrak{h}$-homogeneous. 

Let $K$ be the algebraic closure of ${\rm frac}(R)$. Consider $P(x,y)$ as an element of $K[x,y]$, and view $A$ as a subring of $K\otimes_RA$. By Lemma 4.6 of \cite{Freudenburg.06}, $P$ has the form:
\[
P=x^iy^j\prod_k(\alpha_kx^b+\beta_ky^a)\,\, ,\,\, \alpha_k,\beta_k\in K^*
\]
Since $P(u,v)=0$, it follows that $ru^b+ v^a=0$ for some $r\in K^*$. We see that $r = -u^{-b}v^a\in R^*$. Moreover, $a>1$, since otherwise $v\in R[u,u^{-1}]$. 

Let $m,n\in\Z$ be such that $am+bn=1$. Set $w=u^mv^n$, noting that $w\in A^*$ and $\deg_{\mathfrak{g}}w=d\ne 0$. Then:
\[
w^a=(-r)^nu \quad {\rm and}\quad w^b=(-r)^{-m}v
\]
It follows that:
\[
R[u^{\pm 1},v^{\pm 1}] = R[w^{\pm a},w^{\pm b}] = R[w,w^{-1}]
\]
If $R[A^*]=R[w,w^{-1}]$, the desired result holds. Otherwise, replace $u$ by $w$ and repeat the argument above. Since 
\[
\deg_{\mathfrak{g}}u=ad>d=\deg_{\mathfrak{g}}w>0
\]
this process must terminate in a finite number of steps. This completes the proof of part (a). 

The proof of part (b) is a continuation of the algorithm used in the proof of part (a), where units are adjoined where needed. 

Suppose that $R[A^*]=R[u,u^{-1}]$. If $R[A^*]=A$, there is nothing further to prove. So assume that $R[A^*]\ne A$, and choose $\mathfrak{g}$-homogeneous $v\in A$ not in $R[u,u^{-1}]$. As before, we obtain an equation $ru^b+v^a=0$, where $r\in R$, and $a,b$ are relatively prime integers with $a>1$. However, in this case $r\not\in R^*$, since otherwise $v$ is a unit. 

In order to continue the algorithm, we extend $\mathfrak{g}$ to the ring $A_r$, noting that $v\in A_r^*$.  As above, there exists $w\in A_r^*$ such that 
$0<\deg_{\mathfrak{g}}w<\deg_{\mathfrak{g}}u$ and $R_r[A_r^*]=R_r[w,w^{-1}]$. 

If $A_r=R_r[w,w^{-1}]$, the desired result holds. Otherwise, replace $u$ by $w$ and repeat the argument. As before, since a strict decrease in degrees takes place, the process must terminate in a finite number of steps. This completes the proof of part (b). 
\end{proof}

As a consequence of this theorem, we obtain the following characterization of Laurent polynomial rings over a field. 

\begin{theorem}\label{character} Let $\kk$ be a field, and let $A$ be a $\kk$-algebra. The following are equivalent.
\begin{itemize}
\item [1.] $A=\kk^{[\pm 1]}$.
\smallskip
\item [2.] The following three conditions hold. 
\subitem {\rm (a)} $\kk$ is algebraically closed in $A$
\subitem {\rm (b)} ${\rm tr.deg}_{\kk}A=1$
\subitem {\rm (c)} $A^*\not\subset {\cal N}_{\kk}(A)$
\end{itemize}
\end{theorem}

\begin{corollary}\label{key2} Let $\kk$ be a field, and let $A$ be a $\kk$-algebra. Assume that $\kk$ is algebraically closed in $A$. Given $u\in A^*$, if $u\not\in {\cal N}_{\kk}(A)$, then there exists $w\in A^*$ such that:
\[
{\rm Alg}_{\kk [u]}A=\kk [w,w^{-1}]=\kk^{[\pm 1]}
\]
\end{corollary} 

\begin{proof} By hypothesis, there exists $\mathfrak{g}\in A(\Z ,\kk )$ such that $\deg_{\mathfrak{g}}u\ne 0$. 
Set $B={\rm Alg}_{\kk [u]}A$. Since $u$ is a unit, $u$ is $\mathfrak{g}$-homogeneous, and $\kk [u]$ is a $\mathfrak{g}$-homogeneous subring. By {\it Lemma~\ref{homogeneous}(b)}, it follows that $B$ is $\mathfrak{g}$-homogeneous, meaning that $\mathfrak{g}$ restricts to $B$. Since $\deg_{\mathfrak{g}}u\ne 0$, we see that $u\not\in {\cal N}_{\kk}(B)$. The result now follows from {\it Thm.~\ref{character}}.
\end{proof}


\begin{remark} {\rm In {\it Thm.~\ref{character}}, one cannot generally replace the field $\kk$ by a ring $R$ which is not a field. For example, 
define $A=\kk [u,u^{-1},v]$, where $u,v$ are algebraically independent over $\kk$. Define $\mathfrak{g}\in A(\Z ,\kk)$ by declaring that $\deg_{\mathfrak{g}}u=2$ and $\deg_{\mathfrak{g}}v=1$, and define:
\[
R=A^{\mathfrak{g}}=\kk [u^{-1}v^2]
\]
Then $R$ is algebraically closed in $A$, ${\rm tr.deg}_RA=1$, and $u$ is a unit not in ${\cal N}_R(A)$. However, $A\ne R^{[\pm 1]}$ since the units of $A$ are of the form $\lambda u^n$ for $\lambda\in\kk^*$ and $n\in\Z$.  } 
\end{remark}


\section{Laurent Cancellation}

\subsection{A Reduction}\label{reduce} Let $R$ be an integral domain, and let $A$ be an $R$-algebra. 
If $n\ge 0$, then since $A$ is algebraically closed in $A^{[\pm n]}$ we have: 
\[
{\rm Alg}_R(A^{[\pm n]})={\rm Alg}_R(A)
\]
Let $\alpha :A^{[\pm n]}\to B^{[\pm n]}$ be an isomorphism of $R$-algebras. If $S={\rm Alg}_R(A)$, then $\alpha (S)={\rm Alg}_R(B)$, since $B$ is algebraically closed in $B^{[\pm n]}$. Therefore, identifying $S$ and $\alpha (S)$, we can view $A$ and $B$ as $S$-algebras, and $\alpha$ as an $S$-isomorphism. In considering the question of Laurent cancellation, it thus suffices to assume $R$ is algebraically closed in $A$. Note that this condition implies the group $A^*/R^*$ is torsion free. 

\subsection{Cancellation for Laurent Polynomial Rings}

The following result is due to Bhatwadekar and N. Gupta. 

\begin{theorem}\label{thm-4} {\rm (\cite{Bhatwadekar.Gupta.12}, Lemma 4.5)} Let $R$ be an integral domain, and let $A$ be an $R$-algebra. Suppose that $m,n$ are non-negative integers such that:
\[
A^{[\pm n]}\cong_RR^{[\pm (m+n)]}
\]
Then $A\cong_RR^{[\pm m]}$.
\end{theorem}

\begin{proof} By induction, it suffices to prove the case $n=1$. Let
\[
C=A[y,y^{-1}]=A^{[\pm 1]} \quad {\rm and}\quad D=R[z_1^{\pm 1},...,z_{m+1}^{\pm 1}]=R^{[\pm (m+1)]}
\]
and let $\alpha :C\to D$ be an $R$-isomorphism. We have:
\[
A^*\cdot\Z=C^*=\alpha^{-1}(D^*)=R^*\cdot\Z^{m+1} \quad\Rightarrow\quad  A^*=R^*\cdot\Z^m
\]
If $w_1,...,w_m$ generate the group $A^*/R^*$, then:
\[
D^*/R^*=\langle z_1,...,z_{m+1}\rangle = \langle \alpha (w_1),...,\alpha (w_m),\alpha (y)\rangle
\]
So there exists $E\in SL_{m+1}(\Z)$ such that:
\[
\phi_E(z_i)=\alpha (w_i) \,\, (1\le i\le m)\quad  {\rm and}\quad  \phi_E(z_{m+1})=\alpha (y)
\]
We have:
\[
D=R[z_1^{\pm 1},...,z_{m+1}^{\pm 1}] = R[\phi_E(z_1)^{\pm 1},...,\phi_E(z_{m+1})^{\pm 1}] = R[\alpha (w_1)^{\pm 1},...,\alpha (w_m)^{\pm 1},\alpha (y)^{\pm 1}]
\]
Therefore:
\begin{equation}\label{torus}
A[y^{\pm 1}]=C=\alpha^{-1}(D)=R[w_1^{\pm 1},...,w_m^{\pm 1},y^{\pm 1}]
\end{equation}
Since ${\rm tr.deg}_RC=m+1$, we see that $w_1,...,w_m,y$ are algebraically independent over $R$. 
From line (\ref{torus}), it follows that:
\[
A\cong_RA[y^{\pm 1}]/(y-1)\cong_RR[w_1^{\pm 1},...,w_m^{\pm 1}]=R^{[\pm m]}
\] 
\end{proof}

\subsection{Algebras of Transcendence Degree One over a Field}

For polynomial rings of transcendence degree one over a field, we recall the result of Abhyankar, Eakin, and Heinzer cited in the {\it Introduction}. 

\begin{theorem} {\rm (\cite{Abhyankar.Eakin.Heinzer.72}, 3.3)} Let $\kk$ be a field, and let $A$ be a $\kk$-algebra of transcendence degree one over $\kk$.  Let $B$ be a $\kk$-algebra such that $A^{[n]}\cong_{\kk}B^{[n]}$ for some $n\ge 0$. Then $A\cong_{\kk}B$.
\end{theorem}

For Laurent polynomial rings over rings of transcendence degree one over a field, we have the following.

\begin{theorem}\label{thm-2} Let $R$ be an integral domain, and let $A$ be an $R$-algebra with ${\rm tr.deg}_RA=1$. Then $A$ has the Laurent cancellation property over $R$ if any one of the following conditions holds.
\begin{itemize}
\item [{\bf (a)}] $R[A^*]$ is algebraic over $R$
\medskip
\item [{\bf (b)}] $A^*\subset {\cal N}_R(A)$
\medskip
\item [{\bf (c)}] $R$ is a field
\end{itemize}
\end{theorem}

\begin{proof} By {\it Sect.~\ref{reduce}}, it suffices to assume that $R$ is algebraically closed in $A$. Let $F:A^{[\pm n]}\to B^{[\pm n]}$ be an isomorphism of $R$-algebras.

Part (a): If $R[A^*]$ is algebraic over $R$, then $A^*=R^*$, and $F(A^*)=F(R^*)=R^*\subset B$. 
By {\it Prop.~\ref{key1}}, it follows that $A\cong_RB$ in this case. 

Part (b): Assume that $A^*\subset {\cal N}_R(A)$. By part (a), we may also assume that $R[A^*]$ is transcendental over $R$, meaning that $A$ is algebraic over $R[A^*]$. By {\it Lemma~\ref{neutral}(d)}, ${\cal N}_R(A^{[\pm n]})={\cal N}_R(A)$. Therefore:
\[
F(A^*)\subset F({\cal N}_R(A))=F\left( {\cal N}_R\left( A^{[\pm n]}\right)\right) = {\cal N}_R\left( B^{[\pm n]}\right) \subset {\cal N}_R(B)\subset B
\]
By {\it Prop.~\ref{key1}}, it follows that $A\cong_RB$ in this case. 

Part (c): Assume that $R$ is a field. By part (b), we may assume that $A^*\not\subset {\cal N}_R(A)$. Then, by {\it Cor.~\ref{character}}, we have that $A\cong_RR^{[\pm 1]}$.
By {\it Thm.~\ref{thm-4}}, it follows that $A\cong_RB$ in this case.
\end{proof}


\section{Remarks}\label{summary}

\begin{remark}\label{summary}  {\rm The following two cases of Laurent cancellation are given in \cite{Freudenburg.ppta}. }
\begin{theorem} Let $R$ be an integral domain, and let $A$ be an $R$-algebra. Given $u\in A^*$, set
\[
R_u=\{ \lambda\in R^*\, |\, u-\lambda\in A^*\}\,\, .
\]
Let $A_R^{\tau}$ be the subset of $R$-transtable units, i.e., $u\in A^*$ such that $R_u$ is non-empty. 
If either (a) $R[A^*]=R[A_R^{\tau}]$, or (b) $A$ is algebraic over $R[A_R^{\tau}]$, 
then $A$ has the Laurent cancellation property over $R$.
\end{theorem}
\end{remark}

\begin{remark} {\rm For the $R$-algebra $A$, the question of Laurent cancellation over $R$ is closely related to the question whether ${\cal N}_R(A^{[\pm n]})={\cal N}_R(A)$. In particular, we ask whether this equality holds in the case ${\rm tr.deg}_RA=1$. Another related question is the following: While a $\Z$-grading of $A^{[\pm n]}$  may not restrict to $A$, it does give a $\Z$-filtration of both $A^{[\pm n]}$ and $A$. Let ${\rm Gr}(A^{[\pm n]})$ and ${\rm Gr}(A)$ denote the associated graded rings. Does ${\rm Gr}(A^{[\pm n]})={\rm Gr}(A)^{[\pm n]}$?}
\end{remark}

\bibliographystyle{amsplain}

\bigskip
\bigskip

\noindent \address{Department of Mathematics\\
Western Michigan University\\
Kalamazoo, Michigan 49008}\\
\email{gene.freudenburg@wmich.edu}
\bigskip

\end{document}